\documentclass{amsart}
\usepackage{amsmath,amsthm,amscd,amssymb}
\usepackage{pb-diagram}
\usepackage{verbatim}
\usepackage{graphicx}
\usepackage{amssymb,amscd,enumerate}
\def\a{\alpha}
\def\b{\beta}
\def\cd{\cdot}

\def\g{\gamma}
\def\G{\Gamma}
\def\hra{\hookrightarrow}

\def\lb\{{\left\{}
\def\la{\lambda}
\def\La{\Lambda}
\def\lla{\longleftarrow}
\def\lm{\limits}
\def\lra{\longrightarrow}
\def\dllra{\Longleftrightarrow}
\def\llra{\longleftrightarrow}
\def\n{\nabla}
\def\ngth{\negthickspace}
\def\ola{\overleftarrow}
\def\Om{\Omega}
\def\om{\omega}
\def\op{\oplus}
\def\oper{\operatorname}
\def\oplm{\operatornamewithlimits}
\def\ora{\overrightarrow}
\def\ov{\overline}
\def\ova{\overarrow}
\def\ox{\otimes}
\def\p{\partial}
\def\rb\}{\right\}}
\def\s{\sigma}
\def\sbq{\subseteq}
\def\spq{\supseteq}
\def\sqp{\sqsupset}
\def\supth{{\text{th}}}
\def\T{\Theta}
\def\th{\theta}
\def\tl{\tilde}
\def\thra{\twoheadrightarrow}
\def\un{\underline}
\def\ups{\upsilon}
\def\vp{\varphi}
\def\wh{\widehat}
\def\wt{\widetilde}
\def\x{\times}
\def\z{\zeta}
\def\({\left(}
\def\){\right)}
\def\[{\left[}
\def\]{\right]}
\def\<{\left<}
\def\>{\right>}

\def\tec{Teichm\"uller\ }
\def\sconr{\hbox{\medspace\vrule width 0.4pt height 4.7pt depth
0.4pt \vrule width 5pt height 0pt depth 0.4pt\medspace}}

\def\SA{\mathcal A}
\def\SB{\mathcal B}
\def\SC{\mathcal C}
\def\SD{\mathcal D}
\def\SE{\mathcal E}
\def\SF{\mathcal F}
\def\SG{\mathcal G}
\def\SH{\mathcal H}
\def\SI{\mathcal I}
\def\SJ{\mathcal J}
\def\SK{\mathcal K}
\def\SL{\mathcal L}
\def\SM{\mathcal M}
\def\SN{\mathcal N}
\def\SO{\mathcal O}
\def\SP{\mathcal P}
\def\SQ{\mathcal Q}
\def\SR{\mathcal R}
\def\SS{\mathcal S}
\def\ST{\mathcal T}
\def\SU{\mathcal U}
\def\SV{\mathcal V}
\def\SW{\mathcal W}
\def\SX{\mathcal X}
\def\SY{\mathcal Y}
\def\SZ{\mathcal Z}

\newcommand{\BA}{\ensuremath{\mathbf A}}
\newcommand{\BB}{\ensuremath{\mathbf B}}
\newcommand{\BC}{\ensuremath{\mathbf C}}
\newcommand{\BD}{\ensuremath{\mathbf D}}
\newcommand{\BE}{\ensuremath{\mathbf E}}
\newcommand{\BF}{\ensuremath{\mathbf F}}
\newcommand{\BG}{\ensuremath{\mathbf G}}
\newcommand{\BH}{\ensuremath{\mathbf H}}
\newcommand{\BI}{\ensuremath{\mathbf I}}
\newcommand{\BJ}{\ensuremath{\mathbf J}}
\newcommand{\BK}{\ensuremath{\mathbf K}}
\newcommand{\BL}{\ensuremath{\mathbf L}}
\newcommand{\BM}{\ensuremath{\mathbf M}}
\newcommand{\BN}{\ensuremath{\mathbf N}}
\newcommand{\BO}{\ensuremath{\mathbf O}}
\newcommand{\BP}{\ensuremath{\mathbf P}}
\newcommand{\BQ}{\ensuremath{\mathbf Q}}
\newcommand{\BR}{\ensuremath{\mathbf R}}
\newcommand{\BS}{\ensuremath{\mathbf S}}
\newcommand{\BT}{\ensuremath{\mathbf T}}
\newcommand{\BU}{\ensuremath{\mathbf U}}
\newcommand{\BV}{\ensuremath{\mathbf V}}
\newcommand{\BW}{\ensuremath{\mathbf W}}
\newcommand{\BX}{\ensuremath{\mathbf X}}
\newcommand{\BY}{\ensuremath{\mathbf Y}}
\newcommand{\BZ}{\ensuremath{\mathbf Z}}

\newcommand{\rank}{\operatorname{rank}}

\def\bba{{\mathbb A}}
\def\bbb{{\mathbb B}}
\def\bbc{{\mathbb C}}
\def\bbd{{\mathbb D}}
\def\bbe{{\mathbb E}}
\def\bbf{{\mathbb F}}
\def\bbg{{\mathbb G}}
\def\bbh{{\mathbb H}}
\def\bbi{{\mathbb I}}
\def\bbj{{\mathbb J}}
\def\bbk{{\mathbb K}}
\def\bbl{{\mathbb L}}
\def\bbm{{\mathbb M}}
\def\bbn{{\mathbb N}}
\def\bbo{{\mathbb O}}
\def\bbp{{\mathbb P}}
\def\bbq{{\mathbb Q}}
\def\bbr{{\mathbb R}}
\def\bbs{{\mathbb S}}
\def\bbt{{\mathbb T}}
\def\bbu{{\mathbb U}}
\def\bbv{{\mathbb V}}
\def\bbw{{\mathbb W}}
\def\bbx{{\mathbb X}}
\def\bby{{\mathbb Y}}
\def\bbz{{\mathbb Z}}

\newtheorem{thm}{Theorem}[section]
\newtheorem{lem}[thm]{Lemma}
\newtheorem{cor}[thm]{Corollary}

\newtheorem{prop}[thm]{Proposition}
\newtheorem{defn}[thm]{Definition}
\newtheorem{rem}[thm]{Remark}

\newtheorem{question}[thm]{Question}

\numberwithin{equation}{section}
\numberwithin{figure}{section}

\newcommand{\spec}{\operatorname{spec}}
\newcommand{\Herm}{\operatorname{Herm}}
\newcommand{\GL}{\operatorname{GL}}
\newcommand{\Ker}{\operatorname{Ker}}
\newcommand{\End}{\operatorname{End}}
\newcommand{\Hom}{\operatorname{Hom}}
\newcommand{\Ext}{\operatorname{Ext}}
\newcommand{\Rep}{\operatorname{Rep}}
\newcommand{\Tors}{\operatorname{Tors}}
\newcommand{\dom}{\operatorname{dom}}
\newcommand{\tr}{\operatorname{tr}}
\newcommand{\id}{\operatorname{id}}
\newcommand{\spin}{\operatorname{Spin}}

\newcommand{\Z}{\mathbb{Z}}
\newcommand{\N}{\mathbb{N}}
\newcommand{\C}{\mathbb{C}}
\newcommand{\Q}{\mathbb{Q}}
\newcommand{\K}{\mathbb{K}}
\newcommand{\R}{\mathbb{R}}
\newcommand{\F}{\mathbb{F}}

\newcommand{\RR}{\mathcal{R}}
\newcommand{\UU}{\mathcal{U}}
\newcommand{\NN}{\mathcal{N}}
\newcommand{\DD}{\mathcal{D}}
\newcommand{\KK}{\mathcal{K}}
\newcommand{\FF}{\mathcal{F}}
\newcommand{\MM}{\mathcal{M}}
\newcommand{\LL}{\mathcal{L}}
\newcommand{\CC}{\mathcal{C}}

\newcommand{\QQ}{\mathcal{Q}}
\newcommand{\PP}{\mathcal{P}}

\newcommand{\sra}{\twoheadrightarrow}
\newcommand{\ira}{\rightarrowtail}
\newcommand{\sd}{\rtimes}
\newcommand{\ra}{\longrightarrow}

\newcommand{\lr}{\longleftrightarrow}
\def\x{\times}
\def\p{\partial}
\def\ov{\overline}
\def\Om{\Omega}
\def\s{\sigma}
\def\lra{\longrightarrow}
\renewcommand{\SS}{\mathcal{S}}
\renewcommand{\AA}{\mathcal{A}}
\renewcommand{\l}{\ell}
\renewcommand{\a}{\alpha}
\renewcommand{\i}{\iota}
\renewcommand{\b}{\beta}

\newcommand{\torsionp}{\Z_{(p)}/\Z}
\newcommand{\defeq}{\stackrel{\mathrm{def}}{=}}
\newcommand{\2}{\Z[\pi/\pi^{(2)}]}

\newcommand{\gu}{\G_n^U}
\newcommand{\go}{\G_0^U}
\newcommand{\1}{\pi_1}
\newcommand{\rk}{\text{rank}}

\renewcommand{\dgeverylabel}{\displaystyle}

\newcommand{\hz}[1]{\ensuremath{H_#1(-;\mathbb{Z})}}
\newcommand{\hq}[1]{\ensuremath{H_#1(-;\mathbb{Q})}}
\newcommand{\sn}{\ensuremath{^{(n)}}}
\newcommand{\np}{\ensuremath{^{(n+1)}}}
\newcommand{\gn}{\ensuremath{G^{(n)}_H}}
\newcommand{\an}{\ensuremath{A^{(n)}_H}}

\title{Homological Stability of Series of Groups}
\author{Tim Cochran$^{\dag}$ and Shelly Harvey$^{\dag}$}

\address{Rice University, Houston, Texas, 77005-1892}
\email{cochran@rice.edu, shelly@rice.edu}

\thanks{$^{\dag}$Both authors were partially supported by the National
Science Foundation. The second author was partially supported
by a fellowship from the Sloan Foundation and by an NSF CAREER grant.}

\begin{document}
\begin{abstract} We define the {\it stability} of a {\it subgroup under a class of maps}, and establish the basic properties of this notion. Loosely speaking, we will say that a normal subgroup, or more generally a normal series $\{A_n\}$ of a group $A$, is \textbf{stable} under a class of homomorphisms $\SH$ if whenever $f:A\to B$ lies in $\SH$, we have that $f(a)\in B_n$ if and only if $a\in A_n$. This translates to saying that each element of $\SH$ induces a \emph{monomorphism} $A/A_n\hookrightarrow B/B_n$. This contrasts with the usual theories of localization wherein one is concerned with situations where $f$ induces an \emph{isomorphism}. In the literature, the most commonly considered class of maps are those that induce isomorphisms on (low-dimensional) group homology. The model theorem in this regard is the 1963 result of J. Stallings that (each term of) the lower central series is preserved under any $\bbz$-homological equivalence of groups ~\cite{St}. Various other theorems of this nature have since appeared, involving \emph{different series} of groups- variations of the lower central series. W. Dwyer generalized Stallings' $\bbz$ results to larger classes of maps ~\cite{Dw}, work that was completed in the other cases by the authors. More recently the authors proved analogues of the theorems of Stallings and Dwyer for variations of the \textbf{derived} series ~\cite{CH1}~\cite{CH2}~\cite{CH3}. The above theorems are all different but clearly have much in common. We interpret all of these results in the framework of stability.
\end{abstract}
\maketitle

\section{Introduction}\label{intro}

Loosely speaking we will say that a function, $\Gamma$, that assigns to each group $A$ a normal subgroup $\Gamma(A)\vartriangleleft A$ is \textbf{stable} under a class of maps $\SH$ if whenever $f:A\to B$ lies in $\SH$, we have that $f(a)\in \Gamma(B)$ if and only if $a\in \Gamma(A)$. Precise definitions are given in Section~\ref{stability}. This translates to saying that each element of $\SH$ induces a \emph{monomorphism} $A/\Gamma(A)\hookrightarrow B/\Gamma(B)$. This contrasts with the usual theories of localization wherein one is concerned with situations where $f$ induces an \emph{isomorphism}. The \textbf{stabilization}, $\Gamma_S$, of $\Gamma$ under $\SH$ is the function assigning the ``smallest'' normal subgroup that contains $\Gamma(A)$ and has the property that each $f\in \SH$ induces a monomorphism $A/\Gamma_S(A)\hookrightarrow B/\Gamma_S(B)$.

We are motivated by the general question: ``What subgroups of a group are unchanged, or stable, under homology equivalences''? The model theorem in this regard is the landmark 1963 result of J. Stallings (below) that (each term of) the lower central series of a group $A$ is preserved under any homological equivalence. Recall that $A_n$ is defined recursively by $A_1\equiv A$ and $A_{n+1}=[A,A_n]$.

\begin{thm}\label{StallingsIntegral} ~\cite[Theorem 3.4]{St}(Stallings' Integral Theorem) If $f:A\to B$ is a group homomorphism that induces an isomorphism on $\hz1$ and an epimorphism on $\hz2$ then for each $n$, $f$ induces an isomorphism $A/A_n\cong B/B_n$. Therefore $a\in A_n$ if and only if $f(a)\in B_n$.
\end{thm}

Stallings' Integral Theorem implies that (each term of) the lower central series is stable under the set of $\bbz$-homologically $2$-connected maps.

Stallings had analogous theorems for $\mathbb{Q}$ and $\mathbb{Z}_p$ (our convention is that $\mathbb{Z}_p$ denotes the integers modulo $p$) that involved different series- variations of the lower-central series ~\cite[Theorem 7.3]{St}. William Dwyer improved on Stallings' theorem by weakening the hypothesis on $H_2$ and finding the precise class of maps $f$ which, for a {\it fixed} $n$, yield isomorphisms modulo the $n^{\supth}$ term of the lower central series ~\cite[Theorem 1.1]{Dw}. This was placed in a larger context by A. Bousfield ~\cite{B}.

Dwyer's work implies that the lower central series is stable under a larger class of maps. However, the lower central series \emph{fails} to be stable if one considers $\bbq$-homologically $2$-connected maps, as may be seen by considering the map $\mathbb{Z}_2\to \{e\}$. This situation is remedied by enlarging the lower central series slightly to form the \textbf{rational lower central series} (see Section~\ref{stability}). The rational lower central series is stable under rational homology equivalences as a consequence of Stallings' Rational Theorem (see Section~\ref{stablower}).

More recently the authors found analogues of the theorems of Stallings and Dwyer for the \textbf{derived series}~\cite{CH1}~\cite{CH2}. However to do so it was necessary to expand the derived series and use a larger series, the \textbf{torsion-free-derived series}, introduced by the second author ~\cite[Section 2]{Ha2}. All of these theorems and series will be reviewed as necessary in later sections. The \textbf{derived series} fails dramatically to have the stability property under homological equivalences. For example if we consider the abelianization map below

$$
\mathbb{Z}[t,t^{-1}]/<t^2-t+1> \rtimes \mathbb{Z} \equiv A \overset{f}\lra B\equiv \mathbb{Z}
$$
we see that it is a homologically $2$-connected map. But any non-zero element of the commutator subgroup of $A$, ($\mathbb{Z}[t,t^{-1}]/<t^2-t+1>$), is not itself in the second derived subgroup of $A$, yet maps trivially under $f$ and hence lies in the second term of the derived series of $\mathbb{Z}$. Can the $n$-th term of the derived series be enlarged until it is ``stable'' under homologically-$2$-connected maps? In this case expanding to the ``rational derived series'' fails. Certainly it \emph{can} be so enlarged since we could enlarge it all the way to the $2^n$-th term of the lower-central series which \emph{is} stable under $2$-connected maps. What is the minimum it needs to be enlarged to become stable? These questions motivate the definitions to follow.

\section{Stability of Series of Groups}\label{stability} Let $\SG$ be the category of groups and let $\SC$ be a subcategory.

\begin{defn}\label{subgroup}A \textbf{subgroup function} (short for \textbf{normal subgroup function}) for $\SC$ is a function $\Gamma:\SC\to\SC$, assigning to each $A\in\SC$ a normal subgroup $\G(A)$ of $A$.
\end{defn}

\begin{defn}\label{series} A \textbf{series} (short for \textbf{normal series}) for $\SC$, $\{\G^n\}$, is a collection of subgroup functions $\G^n:\SC\to\SC$, $n\ge0$, such that $\{e\}\subset\dots\subset\G^{n+1}(A)\subset\G^n(A)\dots\subset\G^1(A)\subset\G^0(A)=A$.
\end{defn}

The important examples to keep in mind are the lower central series $\{A_n\}$; the \textbf{rational lower central series}, $\{A^r_n\}$, defined by
$$
A^r_1=A, ~~ A^r_{n+1}=\{x\mid x^k\in[A,A^r_n]~\text{for some positive integer k}\};
$$
and, for a fixed prime $p$, the $\bbz_p$-\textbf{lower central series}, $\{A_{p,n}\}$, (also called the $\textbf{p-lower central series}$ or the $\textbf{lower central p-series}$) which is defined by
$$
A_{p,1}=A, ~~A_{p,n +1}= (A_{p,n})^p[A_{p,n},A].
$$
This is the fastest descending central series whose successive quotients are $\mathbb{Z}_p$-vector spaces ~\cite{St}.

For economy of words we make the following definition.
\begin{defn}\label{def:Rlcs} For $R=\bbz, \bbq$ or $\bbz _p$ where $p$ is prime, respectively, let the $R$-lower central series, $\{A^R_n\}$, be the lower central series, rational lower central series, or $\bbz_p$-lower central series, respectively.
\end{defn}

One also has the \textbf{derived series}, $\{A\sn\}$, given by
$$
A^{(0)}=A, ~~ A^{(n+1)}=[A^{(n)},A^{(n)}];
$$
the \textbf{rational derived series}, $\{A\sn_r\}$, defined by
$$
A^{(0)}=A, ~~~A\np_r=\{x\mid x^k\in[A\sn_r,A\sn_r]~\text{for some positive integer k}\}),
$$
and, the $\bbz_p$-\textbf{derived series}, $\{A^{(n)}_{p}\}$, (also called the $\textbf{p-derived series}$ or the $\textbf{derived p-series}$) which is defined by
$$
A^{(0)}_{p}=A, ~~A^{(n+1)}_{p}= (A^{(n)}_{p})^p[A^{(n)}_{p},A^{(n)}_{p}].
$$
This is the fastest descending series whose successive quotients are $\mathbb{Z}_p$-vector spaces ~\cite{St}.

\begin{defn}\label{classmaps} A \textbf{class of maps}, $\mathcal{H}$, for $\SC$ is a subset of the morphisms of $\SC$ that  contains all isomorphisms, is closed under composition and is closed under \textbf{nudge-outs}, where by the latter we mean that if $f$, $f'\in \mathcal{H}$, as in the diagram below, then there exist $g$ and $g'$ in $\mathcal{H}$ that make the diagram commute.
\begin{equation*}
\begin{CD}
A      @>f>>   B\\
@Vf'VV       @VVVg\\
B'     @>g'>>   C
\end{CD}
\end{equation*}
\end{defn}

Clearly being closed under push-outs implies being closed under nudge-outs.

\begin{defn}\label{invariant} A subgroup function $\G$ is  \textbf{$\SH$-invariant} with respect to the class of maps $\SH$ if whenever $f:A\to B$ is an element of $\SH$ then $f(\G(A))\subset \G(B)$ (This is the same as saying that, for each $n$,  $\G$ is functorial on the category $\SC$ with morphisms restricted to lie in $\SH$.) A series $\{\G^n~|~n\geq 0\}$ is  \textbf{$\SH$-invariant} with respect to a collection of classes of maps $\SH=\{\mathcal{H}^n~|~ n\ge 0\}$
$$
\sbq\SH^n\sbq\dots\sbq\SH^1\sbq\SH^0\subset Morph(\SC)
$$
if whenever $f:A\to B$ is an element of $\SH^n$ then $f(\G^n(A))\subset \G^n(B)$.
\end{defn}

\begin{rem}\label{rem:hinvariance} Each of the versions of the lower central series and derived series defined above, consisting (essentially) of verbal subgroups, is \textbf{$\SH$-invariant} with respect to \textbf{any} class of maps.
\end{rem}

\begin{defn}\label{stabilization} Suppose $\G$ is  an $\SH$-invariant subgroup function. The \textbf{stabilization} of $\G$ with respect to $\SH$, denoted $\G_S$, is the subgroup
$$
\G_S(A)=\{a\in A\mid\exists f:A\to B, f\in\SH, \text{such that} ~~f(a)\in\G(B)\}.
$$
We say that $\G$ \textbf{ is stable under $\SH$} if $\G_S(A)=\G(A)$ for each $A$ in $\mathcal{C}$. Suppose $\{\G^n\}$ is  an $\SH$-invariant series. The \textbf{stabilization} of $\{\G^n\}$ with respect to $\SH=\{\SH^n\}$, denoted $\{\G^n_S\}$, is the series wherein $\G^n_S$ is the stabilization of $\G^n$ with respect to $\SH^n$. We say that the series $\{\G^n\}$ \textbf{ is stable under $\SH$} if each term is stable under the corresponding class of maps.
\end{defn}

\begin{rem}\label{inclusion} If $\SH\subset\wt\SH$ then clearly $\G_{S,\SH}\subset\G_{S,\wt\SH}$.
\end{rem}

We first verify that $\G_S$ is itself a subgroup function.

\begin{prop}\label{normal} If $\G$ is  an $\SH$-invariant subgroup function then $\G_S$ is an $\SH$-invariant subgroup function.  If $\{\G^n\}$  is an $\SH$-invariant series then $\{\G^n_S\}$ is an $\SH$-invariant series.
\end{prop}

\begin{proof}[Proof of Proposition~\ref{normal}] First we show that $\G_S(A)$ is a normal subgroup of $A$. Clearly the identity $e$ lies in $\G_S(A)$. Suppose $a$, $a'\in\G_S(A)$, so there exist $f$, $f'\in\SH$ such that $f(a)\in\G(B)$, $f'(a')\in\G(B')$. Since $\SH$ is closed under nudge-outs, there exists $C$, as shown below, with $g$, $g'\in\SH$.
\begin{equation*}
\begin{CD}
A      @>f>>   B\\
@Vf'VV       @VVVg\\
B'     @>g'>>   C
\end{CD}
\end{equation*}
Since  $\SH$ is closed under composition, $g\circ f=g'\circ f'\in\SH$. Since $\G$ is $\SH$-invariant, we also have
\begin{eqnarray*}
&g'(f'(a'))\subset g'(\G(B'))\subset\G(C)\\
&g(f(a))\subset g(\G(B))\subset\G(C).
\end{eqnarray*}
Thus $(g\circ f)(aa')=(g\circ f)(a)\cd(g\circ f)(a')=g(f(a))\cd g'(f'(a'))$ implying that $(g\circ f)(aa')\in\G(C)$. Thus $aa'\in\G_S(A)$. Moreover if $f(a)\in\G(B)$ then $f(a^{-1})\in\G(B)$ since $\G(B)$ is a subgroup, so $\G_S$ is closed under taking inverses.  Similarly, if $f(a)\in\G(B)$ then $f(xax^{-1})=f(x)f(a)(f(x))^{-1}$ lies in $\G(B)$ since $\G(B)$ is a normal subgroup.  Therefore $\G_S(A)$ is a normal subgroup of $A$.

Next we show that $\G_S$ is $\SH$-invariant. Suppose $f:A\to B$, $f\in\SH$ and $a\in \G_S(A)$. Then there exists $f'\in\SH$ such that $f'(a)\in\G(B')$. Since $\SH$ is closed under nudge-outs, there exists $C$, as below with $g$ and $g'$ in $\SH$ and $g\circ f=g'\circ f'\in\SH$.
\begin{equation*}
\begin{CD}
A      @>f>>   B\\
@Vf'VV       @VVVg\\
B'     @>g'>>   C
\end{CD}
\end{equation*}
Since $\G$ is $\SH$-invariant, $g'(f'(a))\in \G(C)$. Thus $(g\circ f)(a)\in \G(C)$. Since $g\in\SH$, it follows that $f(a)\in \G_S(B)$ as desired.

Finally suppose that $\{\G^n\}$ is an $\SH=\{\SH^n\}$-invariant series. First we show that $\G^{n+1}_S(A)\subset\G^n_S(A)$. Suppose $a\in\G^{n+1}_S(A)$ so there exists $f\in\SH^{n+1}$ such that $f(a)\in\G^{n+1}(B)$ then, by our nesting requirement on $\SH$, $f\in\SH^n$ and $f(a)\in\G^n(B)$, so $a\in\G^n_S(A)$. Finally, if $\G^0(A)=A$ then, since $\G^0(A)\subset \G^0_S(A)$, $A=\G^0(A)=\G^0_S(A)$.
\end{proof}

The following establishes the salient properties of the stabilization of a subgroup function.

\begin{thm}\label{main}  Suppose $\SH$ is a class of maps and $\G$ is an $\SH$-invariant subgroup. Then
\begin{enumerate}
\item[1)] for each $A\in\SC$, $\G(A)$ is a normal subgroup of $\G_S(A)$
\item[2)]  for every $f:A\to B$, $f\in\SH$, $f$ induces a monomorphism $A/\G_S(A)\hra B/\G_S(B)$.
\item[3)] $\G_S$ is the initial $\SH$-invariant subgroup function satisfying 1) and 2) above. To be specific, if $\{\wt\G\}$ is an $\SH$-invariant subgroup function such that:
\begin{enumerate}
\item[3.1)] $\G(A)\subset\wt\G(A)$ $\forall A\in\SC$, and
\item[3.2)] every $f\in\SH$ induces a monomorphism $A/\wt\G(A)\hookrightarrow B/\wt\G(B)$,
\end{enumerate}
\hspace{-.25in} then $\G_S(A)\subset\wt\G(A)$ $\forall A\in\SC$.

\end{enumerate}
\end{thm}

\begin{proof}[Proof of Theorem~\ref{main}]  Given $x\in \G(A)$, since the identity map $i_A$ is in $\SH$, $x\in \G_S(A)$. Normality was established in the proof of Proposition~\ref{normal}. This establishes $1$).

Suppose $f:A\to B$, $f\in\SH$. Since $\G_S$ is $\SH$-invariant by Proposition~\ref{normal}, $f$ induces a map $f:A/\G_S(A)\to B/\G_S(B)$. Now suppose $f(a)\in\G_S(B)$. Then, by definition, there exists $g\in\SH$, $g:B\to C$ such that $g\circ f(a)\in\G(C)$. Since $\SH$ is closed under composition, $g\circ f\in\SH$. Hence $a\in\G_S(A)$, establishing $2$).

Now suppose $\wt\G$ is another $\SH$-invariant series satisfying $3.1$ and $3.2$. Suppose $a\in\G_S(A)$, so there exists $f\in\SH$, $f:A\to B$, such that $f(a)\in\G(B)$. By property $3.1$ for $\wt\G$, $\G(B)\subset \wt\G(B)$, so $f(a)\in \wt\G(B)$. By property $3.2$ for $\wt\G$ $f$ induces a monomorphism
$$
f: A/\wt\G(A)\hra B/\wt\G(B).
$$
Hence $a\in \wt\G(A)$. Thus $\G_S(A)\subset \wt\G(A)$ establishing $3$).
\end{proof}

Now we arrive at our major tool for determining the stabilization.

\begin{cor}\label{cor:stable} If $\G$ is an $\SH$-invariant subgroup function for which every $f\in\SH$ induces a monomorphism $A/\G(A)\lra B/\G(B)$, then, $\forall A\in\SC$, $\G_S(A)=\G(A)$ , that is, $\G$ is stable under $\SH$. If $\{\G^n\}$ is an $\SH$-invariant series for which every $f\in\SH^n$ induces a monomorphism $A/\G^{n}(A)\lra B/\G^{n}(B)$, then $\{\G^n\}$ is stable under $\SH$.
\end{cor}
\begin{proof}[Proof of Corollary~\ref{cor:stable}] Since $\G$ satisfies $3.1$ and $3.2$ above,
$\G_S(A)\subset\G(A)$. By property $1)$ of Theorem~\ref{main}, $\G(A)\subset \G_S(A)$.
\end{proof}

\begin{cor}\label{cor:idempotency}(Idempotency) The stabilization of an $\SH$-invariant subgroup function (or series), with respect to $\SH$, is itself stable with respect to $\SH$.
\end{cor}
\begin{proof}[Proof of Corollary~\ref{cor:idempotency}] By Proposition~\ref{normal} and part $2)$ of Theorem~\ref{main}, $\G_S$ satisfies the hypothesis of Corollary~\ref{cor:stable}. Thus $\G_S$ is itself stable under $\SH$.
\end{proof}

\section{Stabilizations of the $R$-Lower Central Series}\label{stablower}

Let $R=\bbz, \bbq$ or $\bbz _p$ where $p$ is prime. Suppose $\SC$ is the category of all groups and consider $\SH=\SH^{R}$ the set of all homologically \emph{$2$-connected maps} with $R$-coefficients; that is, homomorphisms that induce isomorphisms on $H_1(-;R)$ and epimorphisms on $H_2(-;R)$. It is an easy exercise to show that these classes are closed under push-outs. Let $\G^nA=A_{n+1}^R$, i.e. $\G^n$ is the $(n+1)$-st term of the $R$-lower central series. Results of Stallings establish that maps in $\SH^{R}$ induce monomorphisms modulo any term of the $R$-lower central series (for $R=\mathbb{Z}$ see Theorem~\ref{StallingsIntegral}, for $R=\mathbb{Q}$ see \cite[Theorem~7.3]{St}, for $R=\mathbb{Z}_p$ see \cite[Theorem~3.4]{St}). Therefore, using Corollary~\ref{cor:stable}, these can be reinterpreted as:

\begin{prop}\label{lowercentralstable} The $R$-lower central series is stable with respect to all 2-connected maps with $R$-coefficients, that is, $\G^n_S(A)=\G^n(A)$ under $\SH^{R}$.
\end{prop}

We will show that the $R$-lower central series is stable under a much larger class of maps than $\SH^R$. The work of Dwyer suggested the following filtration of $H_2(A;R)$.

\begin{defn}\label{def:dwyerphi} For any group $A$ and positive integer $n$, and $R=\bbz, \bbq$ or $\bbz _p$, let $\Phi_n^R(A)$ denote the kernel of the natural map
$$
H_2(A;R)\to H_2(A/A_n;R).
$$
\noindent Note that, if $m \geq n$ then $\Phi_m^R(A)\subset\Phi^R_{n}(A)$.
\end{defn}

This in turn suggests the following class of maps.

\begin{defn}\label{def:dwyerRclass} For $R=\bbz, \bbq$ or $\bbz _p$, let $\SH^R_{Dwyer}\equiv\{(\SH^R_{Dwyer})^n\}$ be the class of maps, called the \textbf{Dwyer $R$-class}, wherein a map is in $(\SH^R_{Dwyer})^n$ if it induces  an isomorphism on $H_1(-;R)$ and an epimorphism $H_2(A;R)/\Phi^R_n(A)\to H_2(B;R)/\Phi^R_n(B)$.
\end{defn}

One easily checks that these classes are closed under push-outs.

\begin{thm}\label{thm:lowercentralstable} The $R$-lower central series is stable with respect to Dwyer's $R$-class of maps, $\SH^R_{Dwyer}$.
\end{thm}

\begin{proof} By Remark~\ref{rem:hinvariance}, the $R$-lower central series $\G^n(A)=A^R_{n+1}$ is $\SH^R_{Dwyer}$-invariant.
By Corollary~\ref{cor:stable}, it suffices to show that any map $f:A\to B$ such that $f\in(\SH^R_{Dwyer})^n$ induces a monomorphism
$$
A/A_{n+1}^R\lra B/A_{n+1}^R.
$$
In the case $R=\Z$ this was established by Dwyer \cite[Theorem 1.1]{Dw}. More recently the authors showed this in the cases $R=\Q$ and $R=\Z_p$ ~\cite[Theorem 3.1]{CH2}\cite[Theorem 3.1]{CH3}.
\end{proof}

Despite these positive results there are simple questions that are unanswered:

\begin{question}\label{ques:stablower} Is the stabilization of the lower central series under $\SH^R$ (for $R=\bbz_p$ or $\bbq$) equal to the $R$-lower central series?
\end{question}

We do not know the answer to the above question but we can show:

\begin{prop}\label{prop:stablcsunderR} The stabilization of the lower central series under $\SH^R_{Dwyer}$ is the $R$-lower central series. More generally, the stabilization, under $\SH^R_{Dwyer}$, of any series that is contained in the $R$-lower central series, is the $R$-lower central series.
\end{prop}

This is the first case we have discussed where the stabilization is strictly larger than the original series, and we can calculate the stabilization precisely.

\begin{proof} Suppose $\G^n(A)\subset \wt\G^{n}(A)=A^R_{n+1}$. Thus property $3.1)$ of Theorem~\ref{main} holds. By the above mentioned theorems (\cite[Theorem 1.1]{Dw}\cite[Theorem 3.1]{CH2}\cite[Theorem 3.1]{CH3})
property $3.2)$ also holds. Hence, by Theorem~\ref{main}, $\G^n_S(A)\subset\wt\G^n(A)$.

We must show that $\wt\G^n(A)\subset\G^n_S(A)$. This is trivially true for $n=0$ so assume that $n>0$. Suppose $a\in \wt\G^n(A)=A^R_{n+1}$. Consider the projection map:
$$
f:A\to A/\langle a\rangle\equiv B.
$$
Since $f(a)=1$, $f(a)\in \G^n(B)$. Hence, if we can establish that $f\in(\SH^R_{Dwyer})^n$ then, by definition, $a\in\G^n_S(A)$. Since $n+1\geq 2$, $a\in A^R_{2}$. Note that
$$
H_1(A;\bbz_p)=A/A_{p,2}=A/A^R_2~~\text{when}~~R=\bbz_p,
$$
and
$$
H_1(A;\bbq)=A/A^r_{2}\otimes\bbq=A/A^R_2\otimes\bbq~~\text{when}~~R=\bbq.
$$
Thus in all cases it follows that $f$ induces an isomorphism on $H_1(-;R)$. It is fairly easy to show that, since the kernel of $f$ is contained in $A^R_{n+1}$, $f$ induces an isomorphism
$$
A/A_{n+1}^R\cong B/A_{n+1}^R.
$$
To show that $f\in(\SH^R_{Dwyer})^n$ it now suffices to show that $f$ induces an epimorphism:
$$
H_2(A;R)\to H_2(B;R)/\Phi^R_n(B).
$$
In the cases $R=\Z, \Q, \Z_p$ this follows from ~\cite[Theorem 1.1]{Dw},\cite[Theorem 3.1]{CH2} and \cite[Theorem 3.1]{CH3} respectively, which are converses to the part of those theorems that were utilized above. In the case of \cite[Theorem 3.1]{CH3} this is not explicitly stated, but the reader can see that it follows directly from the last commutative diagram in the proof of \cite[Theorem 3.1]{CH3}.
\end{proof}

\section{The Stabilization of the p-Derived Series}

The derived series in not stable under homological equivalences. However, we find that the $p$-derived series behaves more like the lower central series. The underlying reason for this is that, if $A$ is finitely generated, then $A/A^{(n)}_p$ is a finite $p$-group and hence nilpotent. In this section let $\G^n(A)=A^{(n)}_p$

\begin{prop}\label{prop:derivedp} The $\bbz_p$-lower central series is stable with respect to all 2-connected maps with $\bbz_p$-coefficients between finitely generated groups, that is, $\G^n_S(A)=\G^n(A)$ under $\SH^{\bbz_p}$ (but restricting to finitely generated groups).
\end{prop}

\begin{proof}[Proof of Proposition~\ref{prop:derivedp}] This is a direct consequence of Corollary~\ref{cor:stable} and the recent result of the authors ~\cite[Corollary 4.3]{CH3}.
\end{proof}

In fact, this recent work of the authors allows us to also show that the $\bbz_p$-lower central series is stable with respect to a larger class of maps that induce only \emph{monomorphisms} on $H_1(-;\bbz_p)$. This leaves the realm of the usual homological localization theory, which is concerned with homological \emph{equivalences} (see Section~\ref{sec:homlocal}).

\begin{prop}\label{prop:derivedp2} The $\bbz_p$-lower central series is stable with respect to the class of all maps between finitely generated groups that induce a monomorphism on $H_1(- ;\mathbb{Z}_p)$ and an
epimorphism on $H_2(- ;\mathbb{Z}_p)$.
\end{prop}

We will now show that the $\bbz_p$-lower central series is stable under a much larger class of maps. First we define a filtration of $H_2(A;\bbz_p)$ analogous to Dwyer's but appropriate for the derived series.

\begin{defn}\label{derivedDwyer} For a group $A$ and a non-negative integer $n$, let $\Phi\sn_p(A)$ denote the image of the inclusion-induced map
$$
H_2(A^{(n)}_p;\bbz_p)\lra H_2(A;\bbz_p).
$$
\end{defn}

This in turn suggests the following class of maps.

\begin{defn}\label{def:CHderivpclass} Let $\SH_p^{CH}\equiv\{(\SH_p^{CH})^n\}$ be the class of maps whose elements are maps $f:A\to B$ where $A$ is finitely generated, $B$ is finitely presented and where $f$ induces  an isomorphism on $H_1(-;\bbz_p)$ and an epimorphism $H_2(A;\bbz_p)/\Phi^{(n-1)}_p(A)\to H_2(B;\bbz_p)/\Phi^{(n-1)}_p(B)$.
\end{defn}

One easily checks that this class is closed under push-outs.

\begin{prop}\label{prop:deriveddwywerp} The $\bbz_p$-lower central series is stable with respect the $\SH_p^{CH}$, and in fact is stable with respect to the larger class of maps which induce a monomorphism on $H_1(-;\bbz_p)$ rather than an isomorphism.
\end{prop}

\begin{proof}[Proof of Proposition~\ref{prop:deriveddwywerp}] This is a direct consequence of Corollary~\ref{cor:stable} and ~\cite[Theorem 4.2]{CH3}.
\end{proof}

\section{The Stabilization of the Derived Series}

 The derived series and the rational derived series, unlike the lower central series and the $p$-derived series, are highly unstable under homology equivalences. Indeed, until quite recently, nothing much was known or even suspected about their properties under homological equivalences. In 2003, the second author introduced a superseries of the derived series, the \textbf{torsion-free derived series}, $\gn$ ~\cite[Section 2]{Ha2}, for which the authors were able to prove analogues of the theorems of Stallings and Dwyer [CH1] [CH2]. Here we examine the relationships between the torsion-free derived series and the stabilizations of the ordinary derived series with respect to several natural classes of maps.

We first recall the recursive definition of the torsion-free derived series. Let $G^{(0)}_H\equiv G$. Suppose $\gn$ has been defined as a normal subgroup of $G$. Then $\gn/[\gn,\gn]$ is a right $\bbz[G/\gn]$-module ($G$ acts by conjugation). Since $\bbz[G/\gn]$ is an Ore domain ~\cite[Prop. 2.1]{Ha2}, this module has a well-defined torsion submodule $T$. Then $G\np_H$ is defined to be the inverse image of $T$ under the natural map
$$
\pi: \gn\lra\gn/[\gn,\gn].
$$
An easy induction shows $G\sn\subset G^{(n)}_r\subset\gn$.

We define several natural classes of maps and then investigate the stabilization of the derived series with respect to these classes. Recall that the classes $\SH^\bbz$ and $\SH^\bbq$ have been previously defined as those maps that are homologically 2-connected with $\bbz$ or $\bbq$ coefficients respectively. But in this section we will use these symbols to designate these same classes but restricted to those 2-connected maps $f:A\to B$ wherein $A$ is finitely generated and $B$ is finitely presented.

\begin{defn}\label{def:derivedDwyer} For a group $A$ and a non-negative integer $n$, let $\Phi\sn(A)$ (respectively $\Phi\sn_\bbq(A)$) denote the image of $H_2(A\sn;\bbz)\to H_2(A;\bbz)$ (respectively $H_2(A\sn;\bbq)\to H_2(A;\bbq)$).
\end{defn}

The analogy between this filtration and Dwyer's filtration involving the lower central series will not be apparent. See ~\cite[Section 1]{CH2} for a discussion.

This suggests the following classes of maps.

\begin{defn}\label{def:CHderivclass} Let $\SH^{CH}\equiv\{(\SH^{CH})^n\}$ (respectively $\SH^{CH}_\bbq\equiv\{(\SH^{CH}_\bbq)^n\}$) be the class whose elements are maps $f:A\to B$ where $A$ is finitely generated, $B$ is finitely presented and where $f$ induces  an isomorphism on $H_1(-;\bbz)$ and an epimorphism $H_2(A;\bbq)/\Phi^{(n-1)}(A)\to H_2(B;\bbz)/\Phi^{(n-1)}(B)$ (respectively, an isomorphism on $H_1(-;\bbq)$ and an epimorphism $H_2(A;\bbq)/\Phi^{(n-1)}_\bbq(A)\to H_2(B;\bbq)/\Phi^{(n-1)}_\bbq(B)$).
\end{defn}

One easily checks that these classes are closed under composition and push-outs. It is also clear that $\SH^{\bbz}\subset \SH^{CH}$, $\SH^{\bbq}\subset \SH^{CH}_\bbq$ and $\SH^{CH}\subset \SH^{CH}_\bbq$.

The following gives an ``upper bound'' on the stabilization of the derived series.

\begin{thm}\label{thm:stabderived} Let $\G^n(A)=A\sn$, the derived series. Let $\G^n_S$ be the stabilization of the derived series with respect to $\SH^{CH}_\bbq$. Then $A\sn\subset\G^n_S(A)\subset A\sn_H\cap A^r_{2^n}$. The same holds for stabilization with respect to any class of maps contained in $\SH^{CH}_\bbq$, such as $\SH^\bbz$, $\SH^\bbq$ and $\SH^{CH}$. (Indeed the same holds if enlarge the class $\SH^{CH}_\bbq$ by relaxing the isomorphism condition on $H_1$ to a monomorphism condition).
\end{thm}

\begin{proof}[Proof of Theorem~\ref{thm:stabderived}] It follows from ~\cite[Theorem 2.1]{CH2} that the torsion-free derived series is $\SH^{CH}_\bbq$-invariant, and that any $f\in (\SH^{CH}_\bbq)^n$ induces a monomorphism $A/A^{(n)}_H\hra B/B^{(n)}_H$. Since the torsion-free derived series contains the derived series, we conclude that the  torsion-free derived series satisfies $3.1$ and $3.2$ of Theorem~\ref{main} (letting $\tl\G^n(A)=A^{(n)}_H$) and hence the stabilization of the derived series is no larger than the torsion-free derived series, i.e. $\G^n_S(A)\subset A\sn_H$.

On the other hand, if we consider the series given by $\tl\G^n(A)=A^r_{2^{n}}$, the $2^{n}$-th term of the rational lower central series, then we claim that this satisfies $3.1$ and $3.2$ of Theorem~\ref{main}. Once having shown this,  $\G^n_S(A)\subset \tl\G^n(A)=A^r_{2^{n}}$.

Now we prove the claim. Suppose $f\in (\SH^{CH}_\bbq)^n$. We shall show that this implies that $f\in (\SH_{Dwyer}^\bbq)^{2^n-1}$. Then the claim follows directly from ~\cite[Theorem 3.1]{CH2}.

Since $f\in (\SH^{CH}_\bbq)^n$, by Definitions~\ref{def:CHderivclass} and ~\ref{def:derivedDwyer}, it induces an epimorphism
$$
f_*:H_2(A;\bbq)\to H_2(B;\bbq)/\Phi^{(n-1)}_\bbq(B).
$$
and we need to show (by Definitions~\ref{def:dwyerRclass} and ~\ref{def:dwyerphi}) that it induces an epimorphism
$$
f_*:H_2(A;\bbq)\to H_2(B;\bbq)/\Phi_{2^n-1}^\bbq(B).
$$
Therefore it suffices to show that $\Phi^{(n-1)}_\bbq(B)\subset \Phi_{2^{n}-1}^\bbq(B)$. For this it suffices to show that the composition
$$
H_2(B^{(n-1)};\bbz)\overset{i_*}{\lra} H_2(B;\bbz)\overset{\pi_*}{\lra} H_2(B/B_{2^n-1};\bbz)
$$
is the zero map. Since $B^{(n-1)}\subset B_{2^{n-1}}$, this follows from Lemma~\ref{lem:cool} below (setting $k=2^{n-1}$, then $2k-1=2^n-1$).

\begin{lem}\label{lem:cool} For any group $B$ and integer $k$, the map
$$
H_2(B_k;\bbz)\overset{i_*}{\lra} H_2(B;\bbz)\overset{\pi_*}{\lra} H_2(B/B_{2k-1};\bbz)
$$
is the zero map.
\end{lem}
\begin{proof} [Proof of Lemma~\ref{lem:cool}] Suppose $B_k$ is presented by $<F'~|~R'>$. Of course $i(F')\subset B_k$, but we can choose a generating set $F$ for $B$ so large that $i(F')\subset F_k$. Suppose $B$ is presented by $<F~|~R>$. It follows that $B/B_{2k-1}$ is presented by $<F~|~\ov R>$ where $\ov R=~<R,F_{2k-1}>$. Then consider the following commutative diagram where the vertical maps are isomorphisms by Hopf's theorem ~\cite[Theorem 5.3 p.42]{BR}.

$$
\begin{diagram}\label{diagram1}\dgARROWLENGTH=1.2em
\node{H_2(B_k)} \arrow{e,t}{i_*}
\arrow{s,r}{\cong} \node{H_2(B)} \arrow{s,r}{\cong} \arrow{e,t}{\pi_*}\node{H_2(B/B_{2k-1})}
\arrow{s,r}{\cong}\\
\node{\frac{R'\cap [F',F']}{[R',F']}}\arrow{e,t}{i}\node{\frac{R\cap [F,F]}{[R,F]}}\arrow{e,t}{\pi}\node{\frac{\ov R\cap [F,F]}{[\ov R,F]}}
\end{diagram}
$$
\noindent Now note that $i_*([F',F'])\subset [F_k,F_k]\subset F_{2k}$. But $F_{2k}= [F_{2k-1},F]\subset [\ov R,F]$. Hence $\pi\circ i=0$ and the result follows.
\end{proof}

Finally, for any class $\SH$ contained in $\SH^{CH}_\bbq$, $\G^n_{S,\SH}\subset\G^n_{S,\SH^{CH}_\bbq}$ by Remark \ref{inclusion}.
\end{proof}

We seek to characterize the stabilization of the derived series more precisely. Although the class of maps $\SH^{CH}_\bbq$ may appear be the most natural extension to the derived series of Dwyer's class of maps, the following slightly larger class may be even more natural. Indeed for this class we are able to better characterize the stabilization of the derived series by providing both a ``lower bound'' and ``an upper bound.'' Both of these bounds are related to torsion elements of the module $A\sn/A\np$.

\begin{defn}\label{Harveygropes} For any group $B$ and non-negative integer, let $\Phi\sn_H(B)\sbq H_2(B;\bbq)$ be the image of $H_2(B\sn_H;\bbq)\to H_2(B;\bbq)$, where $B^{(n)}_H$ is the torsion-free derived series.
\end{defn}

\begin{defn}\label{Harveyclass} Let $\SH_H=\{\SH\sn_H\}$ where $\SH\sn_H$ is the set of homomorphisms $f:A\to B$ where $A$ is finitely generated, $B$ finitely-related where $f$ induces an isomorphism on $\hq1$ and induce an epimorphism $H_2(A;\bbq)/\Phi^{(n-1)}_H(A)\to H_2(B;\bbq)/\Phi^{(n-1)}_H(B)$.
\end{defn}

It is easy to see that $\SH_H$ is closed under composition and, using  ~\cite[Proposition 2.3]{CH1}, closed under push-outs.

\begin{defn}\label{cohn} Let $A^{(0)}_C\equiv A$ and let $A\np_C$ be the subgroup {\it generated} by the set of elements $x\in A\sn_C$ that represent torsion elements of the $\bbz[A/A\sn_C]$-module $A\sn_C/[A\sn_C,A\sn_C]$ that are annihilated by some $\g\in\bbz[A/A\sn_C]$ whose image under the augmentation $\bbz[A/A\sn_C]\to\bbz$ is non-zero.
\end{defn}

\begin{prop}\label{cohnprops} $\{A\sn_C\}_{n\ge0}$ is a normal series for $A$ whose successive quotients are $\bbz$-torsion free and, for each $n$,
$$
A\sn\subset A\sn_C\subset A\sn_H\cap A^r_{2^n}.
$$
\end{prop}

\begin{proof}[Proof of Proposition \ref{cohnprops}] The proof is by induction on $n$. Suppose the Proposition holds for all values less than or equal to $n$. An element $x\in A\sn_C$ represents a torsion element as above precisely when there exists $\gamma=\Sigma k_ig_i\in\bbz A$ such that $\Sigma k_i\neq0$ and $[x]*[\g]=0$, which translates to the condition:

\begin{equation}\label{eq:cohn}
\prod_i g^{-1}_i x^{k_i} g_i\in[A\sn_C,A\sn_C].
\end{equation}
Note that if satisfies $x$ ~\ref{eq:cohn} then so does $x^{-1}$. A general element of $A\np_C$ is a product of such $x$. To show that $A\np_C$ is normal in $A$ it suffices to show that, for any $g\in A$, $g^{-1}xg$ also satisfies the condition above. But this is easily seen by setting $h_i=g^{-1}g_i$ and observing that
$$
\prod_i h^{-1}_i(g^{-1}xg)^{k_i} h_i = \prod_i g^{-1}_i x^{k_i} g_i.
$$

Clearly $A\np\subset A\np_C$ by definition. To show $A\np_C\subset A\np_H$ assuming that $A\sn_C\subset A\sn_H$, consider the diagram below:
\begin{equation*}
\begin{CD}
A\sn_C      @>\pi_C>>   A\sn_C/[A\sn_C,A\sn_C]\\
@ViVV       @VVi_*V\\
A\sn_H     @>\pi_H>>   A\sn_H/[A\sn_H,A\sn_H].
\end{CD}
\end{equation*}
It suffices to show that, for any $x\in A\sn_C$ with property ~\ref{eq:cohn}, $i(x)\in A\np_H$. But the image, $\bar\g$, of $\g$ in $\bbz[A/A\sn_H]$ clearly annihilates $\pi_H(i(x))$ and so it is only necessary to remark that $\bar\g\neq0$ since it has non-zero augmentation. Thus $i(x)\in A\np_H$ by definition.

To show that $A\np_C\subset A^r_{2^{n+1}}$, assuming $A\sn_C\subset A^r_{2^n}$, it again suffices to consider a single $x$ satisfying ~\ref{eq:cohn}. We will prove $x\in A^r_{2^{n+1}}$ by another induction. Suppose, by induction, that $x\in A^r_{m+2^n}$ for some $0\leq m\leq 2^n-1$. Certainly this is true for $m=0$ since $x\in A\np_C\subset A\sn_C\subset A^r_{2^n}$ by our other inductive hypothesis. Consider our hypothesis
$$
\Sigma g^{-1}_i x^{k_i} g_i\in [A\sn_C,A\sn_C]\subset[A^r_{2^n},A^r_{2^n}]\subset A^r_{2^{n+1}}\subset A^r_{m+1+2^n}
$$
as a statement in the torsion-free abelian group $A^r_{m+2^n}/A^r_{m+1+2^n}$. Since $x\in A^r_{m+2^n}$,
$$
g^{-1}_i x^{k_i} g_i\equiv x_i^{k_i} ~~\text{modulo} ~~A^r_{m+1+2^n}.
$$
Thus our hypothesis simplifies to
$$
(\Sigma k_i)\cd[x] = 0
$$
implying that $[x]=0$ in this quotient and hence that $x\in A^r_{m+1+2^n}$. Iterating this process yields that $x\in A^r_{2^{n+1}}$.
\end{proof}

\begin{thm}\label{main2} On the class of finitely-presented groups, the stabilization, $A\sn_S$, of the derived series with respect to the Harvey class of maps $\SH_H$ satisfies
$$
A\sn_C\subset A\sn_S\subset A\sn_H
$$
(and each of these series is $\SH_H$-invariant).
\end{thm}

\begin{proof} Assume that $A$ is finitely-presented. The proof is by induction on $n$. The case $n=0$ is clear since each of the series above is defined to be $A$ itself in that case. Supposing that the theorem holds for all integers $\le n$, we establish it for $n+1$.

We again invoke the main theorem of ~\cite{CH2} (in a stronger from than that used previously).

\begin{thm}\label{Dwyer2}{\rm(Cochran-Harvey~~\cite[Theorem 2.1]{CH2})}  If $f:A\to B$ and $f\in\SH\sn_H$ then $f$ induces a monomorphism
$$
A/A\np_H\hra B/B\np_H.
$$
\end{thm}

\begin{cor}\label{HarveyInvariance} $A\np_H$ is $\SH^{(n-1)}_H$-invariant (and hence $\SH\sn_H$-invariant and $\SH\np_H$-invariant).
\end{cor}

\begin{proof}[Proof of Corollary~\ref{HarveyInvariance}] If $f\in\SH^{(n-1)}_H$, $f:A\to B$, then by Theorem \ref{Dwyer2}, $f$ induces a monomorphism $A/A\sn_H\to B/B\sn_H$. It then follows from ~\cite[Proposition 2.3]{CH1} that $f(A\np_H)\subset B\np_H$. Thus the $(n+1)$-st term of the torsion-free derived series is $\SH^{(n-1)}_H$-invariant. Since $\SH\np_H\subset\SH\sn_H\subset\SH^{(n-1)}_H$, the other statements follow immediately.
\end{proof}

It follows that the series $\tilde{\G}^n(A)=A\sn_H$ satisfies properties $3.1$ and $3.2$ of Theorem \ref{main} for $\G^n(A)=A\sn$ with respect to $\SH_H$ and so, by Theorem \ref{main}, $A\sn_S\sbq A\sn_H$.

Now we prove that $A\np_C\subset A\np_S$ assuming that $A/A\np_C\subset A/A\np_S$. It suffices to consider $x\in A\np_C$ that satisfies
~\eqref{eq:cohn} since a general element is a product of such $x$. Consider the projection map:
$$
f:A\to A/<x>\equiv B.
$$
Note that $B$ is also finitely presented. Since $f(x)=1$, $f(x)\in B^{(n+1)}$. Hence, if we can establish that $f\in\SH\np_H$ then, by definition, $x\in A^{(n+1)}_S$ with respect to $\SH_H$. Since $n\geq 1$, $x\in A^r_2$ by Proposition~\ref{cohnprops}, so a multiple of $x$ lies in the commutator subgroup of $A$. Hence $f$ induces an isomorphism on $H_1(-;\bbq)$. Thus it suffices to show that $f$ induces an epimorphism
$$
H_2(A;\bbq)\to H_2(B;\bbq)/\Phi^{(n)}_H(B).
$$
Choose a surjection $\ov\phi:\bar{F}\to A^{(n)}_H$ where $\bar{F}$ is free (a generating set for $A^{(n)}_H$) and extend this to a surjection  $\phi:F\to A$ where $\ov F\subset F$ and $F$ is free (a generating set for $A$). Then say $A$ is presented by $<F~|~R>$ and $B$ is presented by $<F~|~R,~ x>$. The cokernel of
$$
H_2(A;\bbq)\to H_2(B;\bbq)
$$
under Hopf's identifications
$$
\frac{R\cap [F,F]}{[R,F]}\otimes \mathbb{Q}\cong H_2(A;\bbq)\to H_2(B;\bbq)\cong\frac{<R, ~x>\cap [F,F]}{[<R,x>,F]}\otimes \mathbb{Q}
$$
is generated by the class of $x\in F$. Since $g^{-1}_i x^{k_i} g_i\equiv x^{k_i}$ modulo $[<R,x>,F]$,
$$
y=\prod_i g^{-1}_i x^{k_i} g_i\equiv x^{\Sigma k_i}.
$$
Thus $y$ also generates the cokernel of $H_2(A;\bbq)\to H_2(B;\bbq)$ (under the indentifications above). It suffices now to show that $y$ is in the image of $H_2(B_H^{(n)})$. Since $x\in A\np_H$, by ~\cite[Proposition 2.5]{CH1}, the map $f$ above induces an isomorphism
$$
A/A\sn_H\cong B/B\sn_H.
$$
In particular this implies that $A\sn_H\to B\sn_H$ is surjective. Hence $\bar{F}\to A\sn_H\to B\sn_H$ is surjective and we may suppose that $B^{(n)}_H\cong <\bar{F}~|~\bar{R}>$. Hence
$$
H_2(B^{(n)}_H;\bbq)\cong\frac{<\bar{R}>\cap [\bar{F},\bar{F}]}{[<\bar{R}>,\bar{F}]}\otimes \mathbb{Q}.
$$
Since by hypothesis
$$
y\in [A\sn_C,A\sn_C]\subset [A\sn_H,A\sn_H],
$$
$y$ is represented by an element of $[\bar{F},\bar{F}]$. Moreover this element is surely in $\bar{R}$ since $y$ represents the trivial element in $B$. Thus the class represented by $y$ is in the image of $H_2(B_H^{(n)})$.
\end{proof}

\section{Relationship with homological localization of groups}\label{sec:homlocal}

In this section we show that if one considers the class of homologically $2$-connected maps then the stabilization is related to certain homological localizations that were previously in the literature. However, this fact does not seem to assist in calculating the stabilization.

Suppose $R=\Q$ or $\Z$ and let $\mathcal{H}^{R}$ be the set of homomorphisms $f:A\to B$ where $A$ is finitely generated, $B$ is finitely presented, and $f$ induces a  $2$-connected map on $R$-homology. For any finitely generated group $A$ there exists a group $\hat{A}$ and a functorial assignment $\theta:A\to \hat{A}$, called the $R$-\textbf{closure of $A$} ~\cite[Theorem 6.1]{Cha}, with the following properties:

\begin{itemize}
\item [1.] For any $f:A\to B$ in $\SH^R$ $f$ induces an isomorphism $\hat{f}:\hat{A}\cong \hat{B}$, and
\item [2.] For any finitely presented group $A$, there is a sequence of finitely presented groups $A_i$ and maps $h_i\in \SH_R$
\begin{equation}\label{eq:directlimit}
A\overset{h_0}\to A_1\overset{h_1}\to\dots A_{i}\overset{h_i}\to \dots \hat{A}
\end{equation}
such that $\underrightarrow{\text{lim}}~A_i\cong\hat{A}$.
\end{itemize}

\begin{prop}\label{relclosure} Suppose $\{\G^n\}$ is  an $\SH^{R}$-invariant series. Suppose also that $\{\G^n\}$ commutes with direct limits of maps in $\SH^{R}$. Then, for any finitely presented group $A$, the stabilization, $\G^n_S(A)$ with respect to $\SH^R$ is the kernel of
$$
A\overset{\theta_A}\to  \hat{A}\overset{\pi}\to \hat{A}/\G^n(\hat{A}).
$$
\end{prop}
\begin{proof} Let $\wt\G^n(A)$ denote this kernel. We must first verify that $\wt\G^n(A)$ is an $\SH^R$-invariant series. For this it suffices to show that, for any $f:A\to B$ in $\SH^R$,
$f(\widetilde{\G}^n(\hat{A}))\subset \widetilde{\G}^n(\hat{B})$. Suppose $a\in \widetilde{\G}^n(\hat{A})$. Then $\theta_A(a)\in \G^n(\hat{A})$. By ~\ref{eq:directlimit} and since $\{\G^n\}$ commutes with direct limits,
\begin{equation}\label{eq:2}
\G^n(\hat{A})= \underrightarrow{\text{lim}}~\G^n(A_i)
\end{equation}
so there is some $i$ such that $h_i(a)\in \G^n(A_i)$. Now consider
the composition
$$
A_i\to  \hat{A}\overset{\hat{f}}\to\hat{B}.
$$
Since $\hat{B}=\underrightarrow{\text{lim}}B_j$ and $A_i$ is finitely generated, there is some $j$ such that this composition factors through $B_j$ as below
$$
\begin{diagram}\label{diagram2}\dgARROWLENGTH=1.2em
\node{A} \arrow{e,t}{h_i}
\arrow{s,l}{f} \node{A_i} \arrow{e,t}{}
\arrow{s,r}{f_i}\node{\hat{A}}
\arrow{s,r}{\hat{f}}\\
\node{B} \arrow{e,t}{H_j}\node{B_j}\arrow{e,t}{}\node{\hat{B}}
\end{diagram}
$$
Since $h_i,f$ and $H_j$ are in $\SH^R$, $f_i\in \SH^R$. Since $\G^n$ is $\SH^R$-invariant, $f_i(h_i(a))\in \G^n(B_j)$. Thus $H_j(f(a))\in \G^n(B_j)$. Since $\{\G^n\}$ commutes with direct limits of maps in $\SH^{R}$,
$$
\G^n(\hat{B})= \underrightarrow{\text{lim}}~\G^n(B_j).
$$
Hence $\theta_B(f(a))\in \G^n(\hat{B})$ so $f(a)\in \widetilde{\G}^n(\hat{B})$ as desired. Thus $\wt\G^n(A)$ is an $\SH^R$-invariant series.

By ~\ref{eq:2}, $\theta_A(\G^n(A))\subset \G^n(\hat{A})$ so
$$
\G^n(A)\subset \widetilde{\G}^n(A),
$$
and thus $\widetilde{\G}^n$ satisfies $3.1$ of Theorem~\ref{main}. We next show that $\widetilde{\G}^n$ satisfies $3.2$ of that theorem. So suppose $f:A\to B$ lies in $\SH^R$. By property $1$ of the $R-$closure, $\hat{f}$ is an isomorphism. We claim that
$$
\hat{f}(\G^n(\hat{A}))\subset \G^n(\hat{B}).
$$
For, if $\alpha\in \G^n(\hat{A})$ then, as above $\alpha$ is represented by some $\a_i\in \G^n(A_i)$. The argument above then shows that $\hat{f}(\alpha)\in \G^n(\hat{B})$. Hence $\hat{f}$ induces
a map
$$
\hat{f}_n:\hat{A}/\G^{n}(\hat{A})\lra \hat{B}/\G^{n}(\hat{B}).
$$
It will follow that this map is an isomorphism, if we can verify that $(\hat{f})^{-1}$ induces a map on these same quotients (going the other way). This is accomplished by establishing that $(\hat{f})^{-1}$ is necessarily induced by a family of maps $B_j\to A_{i_j}$, which are in $\SH^R$, and then proceeding as above. Therefore $\hat{f}_n$ is an isomorphism. Finally consider $[a]$ in the kernel of
$$
A/\widetilde{\G}^n(A)\overset{f_n}\to B/\widetilde{\G}^n(B).
$$
By definition, $f(a)\in \widetilde{\G}^n(B)$ implies that $\pi_B(\theta_B(f(a)))=0$. From the diagram below and the fact that  $\hat{f}_n$ is an isomorphism, it follows that $a\in \widetilde{\G}^n(A)$. Thus $f_n$ above is a monomorphism and so
$$
\begin{diagram}\label{diagram3}\dgARROWLENGTH=1.2em
\node{A} \arrow{e,t}{\theta_A}
\arrow{s,l}{f} \node{\hat{A}} \arrow{s,r}{\hat{f}} \arrow{e,t}{\pi_A}\node{\hat{A}/\G^{n}(\hat{A})}
\arrow{s,r}{\hat{f}_n}\\
\node{B}\arrow{e,t}{\theta_B}\node{\hat{B}}\arrow{e,t}{\pi_B}\node{\hat{B}/\G^{n}(\hat{B})}
\end{diagram}
$$
$\widetilde{\G}^n$ satisfies $3.2$ of Theorem~\ref{main}. By that theorem then
$$
\G^n_S(A)\subset\widetilde{\G}^n(A).
$$

Finally, we show that
$$
\widetilde{\G}^n(A)\subset \G^n_S(A),
$$
which complete our proof that $\widetilde{\G}^n= \G^n_S$. Suppose $a\in \widetilde{\G}^n(A)$ so $\theta_A(a)\in \G^n(\hat{A})$. We saw earlier in the proof that this implies that there is some $i$ such that $h_i(a)\in \G^n(A_i)$. Since $h_i\in\SH^R$, by definition, $a\in \G^n_S(A)$.
\end{proof}

\bibliographystyle{plain}
\bibliography{mybib7mathscinet}
\end{document}